\documentclass[12pt,a4paper]{article}

\usepackage{epsf,epsfig,amsfonts,amsgen,amsmath,amstext,amsbsy,amsopn,amsthm,amsfonts,amssymb,amscd
}
\usepackage{ebezier,eepic}
\usepackage{color}
\usepackage{multirow}
\newtheorem{thm}{Theorem}

\newtheorem{prop}{Proposition}
\newtheorem{lem}{Lemma}
\newtheorem{false statement}{False statement}

\theoremstyle{definition}

\newtheorem{prob}{Problem}

\newcounter{mathitem}
  {\begin{list}{{$(\roman{mathitem})$}}{
   \setcounter{mathitem}{0}
   \usecounter{mathitem}
   \setlength{\topsep}{0pt plus 2pt minus 0pt}
   \setlength{\parskip}{0pt plus 2pt minus 0pt}
   \setlength{\partopsep}{0pt plus 2pt minus 0pt}
   \setlength{\parsep}{0pt plus 2pt minus 0pt}
   \setlength{\leftmargin}{35pt}
   \setlength{\itemsep}{0pt plus 2pt minus 0pt}}}
  {\end{list}}

\begin{document}
\title{\bf\Large On the sum of $k$-th largest distance eigenvalues of graphs}

\date{}

\author{
Huiqiu Lin \footnote{Supported
by the National Natural Science Foundation of China (Nos. 11401211 and 11471121) and Fundamental Research Funds for the Central
Universities (No. 222201714049). ~~~~~~~~~~~~~~~~~~~~~~~~~~~~~~E-mail: huiqiulin@126.com (H.Q. Lin)}\\
{\footnotesize Department of Mathematics,
East China University of Science and Technology,}\\
{\footnotesize Shanghai 200237, P.R.~China}
}
\maketitle
\vspace{-9mm}
\begin{abstract}
\noindent For a connected graph $G$ with order $n$ and an integer $k\geq 1$, we
denote by $$S_k(D(G))=\lambda_1(D(G))+\cdots+\lambda_k(D(G))$$ the sum of $k$
largest distance eigenvalues of $G$. In this paper, we consider the sharp upper bound
and lower bound of $S_k(D(G))$. We determine the sharp lower bounds of $S_k(D(G))$
when $G$ is connected graph and is a tree, respectively, and characterize
both the extremal graphs. Moreover, we conjecture that the upper bound is attained
when $G$ is a path of order $n$ and prove some partial result supporting the
conjecture. To prove our result, we obtain a sharp upper bound of $\lambda_2(D(G))$
in terms of the order and the diameter of $G$, where $\lambda_2(D(G))$ is the second
largest distance eigenvalue of $G$. As applications, we prove
a general inequality involving $\lambda_2(D(G))$, the independence number of $G$,
and the number of triangles in $G$. An immediate corollary
is a conjecture of Fajtlowicz, which was confirmed in \cite{L15-L} by a different
argument. We conclude this paper with some open problems for further study.

\end{abstract}


\date{}

\section{Introduction}
Throughout this paper, we consider simple, undirected
and connected graphs. Let $G$ be a graph with vertex
set $V(G)=\{v_1,\,v_2,\ldots,\,v_n\}$ and edge set $E(G)$,
where $|V(G)|=n$ and $|E(G)|=m$. Let $N(v)$ denote the
neighbor set of $v$ in $G$. Let $S\subset V(G)$. We use
$G[S]$ to denote the subgraph of $G$ induced by $S$. The
\emph{distance} between vertices $v_i$ and $v_j$, denoted
by $d_G(v_i,\,v_j)$, is the length of a shortest path from
$v_i$ to $v_j$ in $G$. The \emph{diameter} of a graph is
the maximum distance between any two vertices of $G$.

Throughout this note, we use $M(G)$ to denote a real symmetric
matrix respect to a connected graph $G$. We use
$\lambda_1(M(G))\geq \lambda_2(M(G))\geq \cdots\geq \lambda_n(M(G))$
to denote all eigenvalues of $M(G)$ and denote by
$$
S_k(M(G))=\lambda_1(M(G))+\cdots+\lambda_k(M(G)),
$$
where $k\geq 1$ is an integer. The \emph{distance matrix}
of $G$, denoted by $D(G)$ (or simply by $D$), is the real symmetric
matrix with $(i,j)$-entry being $d_G(v_i,v_j)$
$(\mbox{or}~d_{ij})$. The \emph{distance eigenvalues}
(resp., \emph{distance spectrum}) of $G$, are denoted by
$$\lambda_1(D(G))\geq \lambda_2(D(G))\geq \cdots\geq \lambda_n(D(G)).$$
Recently, the distance matrix of a graph has received
increasing attention. Aouchiche and Hansen \cite{AH16}
and Lin, Das and Wu \cite{LDW16} proved some results on
the relations between the distance eigenvalues and some
graphic parameters. Lin \cite{L15-D} proved an upper bound
on the least distance eigenvalue of a graph in terms of
it order and diameter. By using some graph operations, Pokorn\'{y}, H\'{i}c, Stevanovi\'{c} and Milo\v{s}evi\'{c} \cite{PHSM}
obtained many infinite families of distance integral graphs.
Very recently, Lu, Huang and Huang \cite{LHH}
characterized all graphs with exactly two distance eigenvalues
different from -1 and -3. Huang, Huang and Lu \cite{HHL} characterized all graphs with exactly three distance eigenvalues
different from -1 and -2.
For more results on the distance matrix
of graphs, we refer the reader to the survey \cite{AH14}.

Our main motivations of this note come from two aspects. The first
one is a paper of Mohar \cite{M09}, in which he proved that
$S_k(A(G))\leq\frac{1}{2}(\sqrt{k}+1)n$ for any graph of order
$n$ and an integer $k\geq 1$. His theorem is
originally motivated by a result of Gernert which states that
$S_2(A(G))\leq n$ for any regular graphs $G$ and the upper
bound is best possible. Another motivations are
the Grone-Merris conjecture \cite{GM1,GM2} and Brouwer's conjecture \cite{BH}.
For any graph $G$ on $n$
vertices with degree sequence $d_1\geq \cdots\geq d_n$,
its conjugate degree sequence is defined
as the sequence $d_1'\geq d_2'\geq\cdots\geq d_n'$ where
$d_k' :=|\{v_i : d_i\geq k\}|.$
The Grone-Merris conjecture, which is proved by Bai \cite{Bai11},
states that and for any $k\in \{1,\ldots, n\}$,
$S_k(L(G))\leq \sum_{i=1}^kd'_i$.
Brouwer's conjecture says that $S_k(L(G))\leq m+\binom{k}{2}$
holds for any simple graph $G$ of order $n$ and size $m$ and any $1\leq k\leq n$.
These topics have received much attentions and Brouwer's conjecture
is widely open now. However, till to now, it seems to have no study
on upper and lower bounds of $S_k(D(G))$.

In this paper, we try to bound it in terms of some parameters of graphs.
We will study the following general problem which
seems interesting and non-trivial.

\begin{prob}
For a connected graph $G$ and an integer $k\geq 2$,
to give tight upper and lower bounds on $S_k(G)$ and to
characterize the extremal graphs corresponding to them,
respectively.
\end{prob}

We first give a lower bound of $S_k(D(G))$.

\begin{thm}\label{ThGeneralSk(G)}
Let $k\geq 2$ be an integer and $n$ be sufficiently large with
respect to $k$. Let $G$ be a connected graph of order $n$.

\noindent$(i)$ Then $S_{k}(D(G))\geq n-k$ where the equality holds if
and only if $G\cong K_n$.

\noindent $(ii)$ If $G$ is a tree, then $S_{k}(D(G))\geq 2n-2k$
where the equality holds if and only if $G\cong K_{1,n-1}.$
\end{thm}

We then consider the upper bound of $S_k(D(G))$. The following problem
is our original motivation.

\begin{prob}\label{ProbUpperPath}
Let $G\ncong P_n$ be a connected graph with order $n$.  For an integer
$k\geq2$ and sufficiently large $n$ with respect to $k$,
does there hold $S_{k}(D(G))<S_{k}(D(P_n))?$
\end{prob}

Very interesting for us, in order to prove some results
supporting this problem, we need to obtain a sharp upper bound on $\lambda_2(D)$,
which may be of its own interest.
\begin{thm}\label{thLambda2Upper}
Let $G$ be a connected graph of order $n$ with diameter $d$.
Then $\lambda_{2}(D(G))\leq \frac{n(d-1)}{2}-d$,
where the equality holds if and only if
$G\cong K_n$ or $G\cong K_{\frac{n}{2},\frac{n}{2}}$.
\end{thm}

As an application of Theorem \ref{thLambda2Upper}, we can prove
the following result. In particular, we can reprove a conjecture by
Fajtlowicz \cite{F98}, which was confirmed in \cite{L15-L}.

\begin{thm}\label{ThIndepTri}
Let $G$ be a connected graph of order $n>s^3+s^2-2s+1$, where $s\geq 2$.
Suppose that the independence number $\alpha(G)\leq s$. Then there hold:

\noindent $(i)$ $\lambda_2(D(G))< 3s^3\cdot\frac{t(G)}{m}$.

\noindent $(ii)$ \rm{(Lin \cite[Theorem 1.2]{L15-L})} If $t=2$, then $\lambda_2(D(G))<t(G)$,
where $t(G)$ denotes the number of triangles in $G$.
\end{thm}

If there is some information on the diameter of a connected graph,
we can prove the following result supporting Problem \ref{ProbUpperPath}
affirmatively.

\begin{prop}\label{ThSkDiameter}
Let $G$ be a connected graph with order $n$ and diameter $d.$
For an integer $k\geq2$ and sufficiently large $n$ with respect
to $k$, if $d<\frac{2n}{3(k+2)}$ then
$S_{k}(D(G))<S_{k}(D(P_n))$.
\end{prop}

\section{Preliminaries}
In this section, we will list some preliminaries and prove some lemmas.
Our one main tool is Cauchy Interlacing Theorem.
\begin{thm} [Cauchy Interlacing Theorem]\label{ThCauchy}
Let $A$ be a Hermitian matrix with order $n$
and let $B$ be a principal submatrix of $A$
with order $m$. If $\lambda_1(A)\geq \lambda_2(A)\geq \cdots \geq
\lambda_n(A)$ are the eigenvalues of $A$ and $\mu_1(B)\geq
\mu_2(B)\geq \cdots \geq \mu_m(B)$ are the eigenvalues of $B$, then
$\lambda_{n-m+i}(A)\leq \mu_i(B)\leq \lambda_i(A)$ for $i=1,
\ldots, m$.
\end{thm}

Another tool is the famous Ramsey Theorem, which has already turned
out to be powerful for problems in spectral graph theory.
For example, see \cite{ZC14} due to Zhang and Cao.

\begin{thm}[Ramsey \cite{R30}]\label{ThRamsey}
Given any positive integers $k$ and $l$, there exists
a smallest integer $R(k,l)$ such that every graph on $R(k, l)$
vertices contains either a clique of $k$ vertices or an independent
set of $l$ vertices.
\end{thm}

The third one is a theorem due to Merris, which helps us
to obtain bounds of distance eigenvalues.
\begin{thm}[Merris \cite{M90}]\label{ThMerrisTree}
Let $G$ be a tree of order $n$. Let
$\lambda_1(D(G))\geq\cdots\geq \lambda_n(D(G))$
be the eigenvalues of $D(G)$ and let
$\mu_1\geq \mu_2\geq\cdots\geq\mu_{n-1}\geq0$
be the eigenvalues of $L(G)$. Then
$$
0>\frac{-2}{\mu_1}>\lambda_2(D(G))\geq\frac{-2}{\mu_2}\geq\cdots\geq\lambda_{n-1}(D(G))\geq\frac{-2}{\mu_{n-1}}\geq\lambda_{n}(D(G)).
$$
\end{thm}

\begin{lem}\label{Lediamorder}
Let $G$ be a graph of order $n$. If $\Delta(G)\leq l$ and
$diam(G)\leq d$, then
$$
n\leq 1+l+l(l-1)+l(l-1)^2+\cdots+l(l-1)^{d-1}.
$$
\end{lem}

Now we shall give the proof of Lemma
\ref{LeLambdak-2}, whose proof relies on Theorems \ref{ThCauchy},
\ref{ThRamsey} and \ref{ThMerrisTree}. Part of technique is inspired
by Zhang and Cao \cite{ZC14}.

\begin{lem}\label{LeLambdak-2}
Let $G$ be a connected graph of order $n$. For any
integer $k\geq 2$, if $n$ is sufficiently large with
respect to $k$ then $\lambda_{k}(D(G))\geq -2$.
\end{lem}

\begin{proof}
We divide the proof into two cases.

\vspace{3mm}
\noindent{\bf Case 1.} $\Delta(G)\geq R(k-1, k-1)$

\vspace{2mm} Let $v\in V(G)$ with $d_G(v)=\Delta(G)$. By Theorem
\ref{ThRamsey}, $G':=G[N(v)]$ either contains a clique $A$ of
size $k-1$ or an independent set $B$ of size $k-1$.

If the first case occurs, then $H=G[A\cup
\{v\}]\cong K_{k}$. By Theorem \ref{ThCauchy} and the inequality
$k\geq 2$,
$\lambda_k(D(G))\geq\lambda_{k}(D(H))=-1$.

If the second case occurs, then
$H=G[B\cup\{v\}]\cong K_{1, k-1}$. By Theorem \ref{ThCauchy}
and the inequality $k\geq 2$, $\lambda_k(D(G))\geq \lambda_{k}(D(H))=-2$.

\vspace{3mm}\noindent{\bf Case 2.} $\Delta(G)< R(k-1, k-1)$

\vspace{2mm} Take $l=R(k-1, k-1)-1$ and $d\geq 2k$. By
Lemma \ref{Lediamorder}, if
$n\geq 1+l+l(l-1)+l(l-1)^2+\cdots+l(l-1)^{d-1}$,
then $diam(G)\geq d+1$. Thus, the distance matrix $D'$
of $P_d$ is a principle submatrix of $D$. Note that
$\mu_i(P_d) = 2 + 2 \cos\frac{i\pi}{d}$ for $i = 1,\cdots, d.$
Thus, by Theorems \ref{ThCauchy} and \ref{ThMerrisTree},
we have $\lambda_{k}(D(G))\geq \lambda_{k}(D(P_d))\geq\frac{-2}{\mu_k(P_d)}\geq-1~~\mbox{for $2k\leq d$}$.
The proof is complete.
\end{proof}

The following three theorems are used in the proof of Proposition \ref{ThSkDiameter}.
\begin{thm}[Merris \cite{M90}]\label{ThMerrisTDia}
Let $T$ be a tree with diameter $d$. Then
$\lambda_{\lfloor\frac{d}{2}\rfloor}(D(T))>-1$.
\end{thm}

\begin{thm}[Zhou and Ili\'c \cite{ZI10}]\label{ThZI}
Let $G$ be a connected graph on $n$ vertices with diameter $d$,
minimum degree $\delta_1$ and the second minimum degree
$\delta_2$. Then
$$\lambda_1(D(G))\leq \sqrt{[dn-\frac{d(d-1)}{2}-1-\delta_1(d-1)][dn-\frac{d(d-1)}{2}-1-\delta_2(d-1)]},$$
where the equality holds if and only if $G$ is a regular graph
with $d\leq 2$.
\end{thm}

\begin{thm}[Ruzieh and Powers \rm{\cite[Corollary 2.2]{RP90}}]\label{ThRP}
The distance spectral radius of the path $P_n$ is
$\lambda_1(D(P_n))=\frac{n^2}{2a^2}-\frac{2+a^2}{6a^2}+O(\frac{1}{n^2})$,
where a is the root of a $\tanh a=1.$ $(a\doteq 1.199679.)$
\end{thm}

Finally, we prove an easy but useful fact to conclude this section.

\begin{lem}\label{lem3.1}
Let $G=G[V_1,V_2]$ be a connected bipartite graph with $|V_1|=r$ and $|V_2|=n-r$ and $e(G)=m$.
Then $$\lambda_1(D(G))\geq \frac{2(n^2+(r-1)n-r^2-2m)}{n}.$$
\end{lem}

\begin{proof}
Note that
\begin{eqnarray*}
W(G)&\geq& m+2\left({r\choose{2}}+{n-r\choose{2}}\right)+3(r(n-r)-m)\\
&=&r(r-1)+(n-r)(n-r-1)+3r(n-r)-2m\\
&=&n^2+(r-1)n-r^2-2m.
\end{eqnarray*}
Then the result follows from that
$\lambda_1(D(G))\geq \frac{2W}{n}.$
\end{proof}

\section{Proofs}

\smallskip
\noindent
{\bf Proof of Theorem \ref{ThGeneralSk(G)}.}
(i) By a simple calculation, we have $S_k(D(K_n))=n-k$
for $k\geq 1$. Let $G\ncong K_n$ be a connected
graph with order $n$. Then
\begin{eqnarray*}
\lambda_1(D(G))\geq \frac{2W(G)}{n}\geq \frac{2\left[m+2({{n}\choose{2}}-m)\right]}{n}=2(n-1)-\frac{2m}{n}.
\end{eqnarray*}
If $m\leq \frac{n(n-k)}{2}$, then $\lambda_1(D(G))\geq n+k-2.$
Note that $\lambda_2(D(G))\geq -1$. By Lemma \ref{LeLambdak-2},
if $n$ is sufficiently large with respect to $k$, then $\lambda_k(D(G))\geq -2$.
So we obtain
\begin{eqnarray*}
S_k(D(G))&=&\lambda_1(D(G))+\lambda_2(D(G))+\cdots+\lambda_k(D(G))\\
&\geq& n+k-2-1-2(k-2)\\
&=&n-k+1\\
&>&n-k\\
&=&S_k(D(K_n)).
\end{eqnarray*}
If $m> \frac{n(n-k)}{2}$, then $m>(1-\frac{1}{k})\frac{n^2}{2}$
when $n\geq k^2$ (recall that $n$ is sufficiently large with
respect to $k$). By Tur\'an's theorem, $G$ contains a $K_k$. Thus, we have
\[\lambda_i(D(G))\geq \lambda_i(D(K_k))=-1 \ \  \mbox{for $i=2,\ldots,k.$}\]
Since $G\neq K_n$, we obtain $\lambda_1(D(G))>n-1$.
It follows that $S_k(D(G))>n-1-(k-1)=n-k=S_k(D(K_n))$. If $G=K_n$,
it is easy to find that $S_k(D(G))=n-k$.
This completes the proof. {\hfill$\Box$}

\vspace{3mm}

(ii) Let $T\ncong K_{1,n-1}$.  Similar to the proof of Lemma \ref{LeLambdak-2},
we have either $\Delta(T)\geq k-1$ or $diam(T)\geq 2k+1$, where $k\geq 2$. It follows
that either $K_{1,k-1}\subset G$ or $P_{2k+2}\subset G$.
Then by Theorem \ref{ThCauchy} and Lemma \ref{LeLambdak-2},
either $\lambda_k(D(T))\geq -2$ or $\lambda_k(D(T))>-1.$
Thus, we have
$\lambda_3(D(T))+\cdots+\lambda_k(D(T))\geq -2(k-2)$.
Note that $diam(T)\geq 3$. Then there exists a bipartite
partition of $T=T(V_1,V_2)$ such that $|V_1|=r$ and
$|V_2|=n-r$ with $2\leq r\leq \frac{n}{2}.$
Then by Lemma \ref{lem3.1}, we have
\begin{eqnarray*}
\lambda_1(D(T))&\geq&\frac{2n^2+2(r-1)n-2r^2-4m}{n}\\
&=& \frac{2n^2+2(r-1)n-2r^2-4(n-1)}{n}\\
&\geq&\frac{2n^2+2(2-1)n-2\times2^2-4(n-1)}{n}\\
&=&2n-2-\frac{4}{n}\\
&>&2n-3.
\end{eqnarray*}
Combing with $\lambda_2(D(T))\geq-1$, we have
$$S_k(T)>2n-3-1-2(k-2)=2n-2k=S_k(K_{1,n-1}).$$
If $T=K_{1,n-1}$, then $S_k(T)=2n-2k$.
This completes the proof. {\hfill$\Box$}

\vspace{3mm}

\noindent
{\bf Proof of Theorem \ref{thLambda2Upper}.}
If $d=1$, then $G\cong K_n$ and $\lambda_2(D(G))=-1$,
hence the result holds. In the following, set $\lambda_2(D)=\lambda_2(D(G))$.
Assume that $d\geq 2.$
Let $X$ be an eigenvector of $D(G)$ corresponding to $\lambda_{2}(D)$.
We use $x_{v}$ to denote the entry of $X$ corresponding to the vertex $v\in V(G)$. Define
$S^{+}=\{v\in V(G):x_{v}>0\}$ and $S^{-}=\{v\in V(G):x_{v}<0\}$.
For $v\in S^{+}$, we have
\begin{eqnarray*}
\lambda_{2}(D)x_{v}&=&\sum_{u\in S^{+}\backslash \{v\}}d(u,v)x_{u}+\sum_{w\in S^{-}}d(w,v)x_{w}\leq d\sum_{u\in S^{+}\backslash \{v\}}x_{u}+\sum_{w\in S^{-}}x_{w},
\end{eqnarray*}
that is,
$(\lambda_{2}(D)+d)x_{v}\leq d\sum_{u\in S^{+}}x_{u}+\sum_{w\in S^{-}}x_{w}$.
So
$$(\lambda_{2}(D)+d)\sum_{v\in S^{+}}x_{v}\leq |S^{+}|d\sum_{u\in S^{+}}x_{u}+|S^{+}|\sum_{w\in S^{-}}x_{w},$$
that is,
\begin{equation}\label{eq1}
(\lambda_{2}(D)+d-d|S^{+}|)\sum_{v\in S^{+}}x_{v}\leq |S^{+}|\sum_{w\in S^{-}}x_{w}.
\end{equation}
For $v\in S^{-}$, we have
\begin{eqnarray*}
\lambda_{2}(D)x_{v}&=&\sum_{u\in S^{+}}d(u,v)x_{u}+\sum_{w\in S^{-}\backslash \{v\}}d(w,v)x_{w}\geq \sum_{u\in S^{+}}x_{u}+d\sum_{w\in S^{-}\backslash \{v\}}x_{w}.
\end{eqnarray*}
Similarly,
\begin{equation}\label{eq2}
(\lambda_{2}(D)+d-d|S^{-}|)\sum_{v\in S^{-}}x_{v}\geq |S^{-}|\sum_{w\in S^{+}}x_{w}.
\end{equation}
Combining Eqs. (\ref{eq1}) and (\ref{eq2}), we obtain
$$(\lambda_{2}(D)+d-d|S^{+}|)(\lambda_{2}(D)+d-d|S^{-}|)\sum_{v\in S^{+}}x_{v}\sum_{u\in S^{-}}x_{u}\leq |S^{+}||S^{-}|\sum_{v\in S^{+}}x_{v}\sum_{u\in S^{-}}x_{u},$$
that is,
$$(\lambda_{2}(D)+d-d|S^{+}|)(\lambda_{2}(D)+d-d|S^{-}|)\geq |S^{+}||S^{-}|.$$
Let
$$f(y)=(y+d-d|S^{+}|)(y+d-d|S^{-}|)-|S^{+}||S^{-}|.$$
Clearly, the roots of $f(y)=0$ are
$$y_{1}=\frac{d(|S^{+}|+|S^{-}|-2)+\sqrt{d^{2}(|S^{+}|+|S^{-}|)^{2}-4(d^2-1)|S^{+}||S^{-}|}}{2}$$
and
$$y_{2}=\frac{d(|S^{+}|+|S^{-}|-2)-\sqrt{d^{2}(|S^{+}|+|S^{-}|)^{2}-4(d^2-1)|S^{+}||S^{-}|}}{2},$$
respectively.
From Eqs. (\ref{eq1}) and (\ref{eq2}), we can see that $\lambda_{2}(D)\leq d|S^{+}|-d$ and $\lambda_{2}(D)\leq d|S^{-}|-d$. Hence,
$y_{1}>\frac{d(|S^{+}|+|S^{-}|-2)}{2}\geq \min\{d|S^{+}|-d,d|S^{-}|-d\}\geq \lambda_{2}(D)$.
Since $f(\lambda_{2}(D))=(\lambda_2(D)-y_1)(\lambda_2(D)-y_2)\geq 0$, we have
\begin{eqnarray*}
\lambda_{2}(D)&\leq& y_{2}\\
&=&\frac{d(|S^{+}|+|S^{-}|-2)-\sqrt{d^{2}(|S^{+}|+|S^{-}|)^{2}-4(d^2-1)|S^{+}||S^{-}|}}{2}\\
&\leq& \frac{d(|S^{+}|+|S^{-}|-2)-(|S^{+}|+|S^{-}|)}{2}\\
&\leq &\frac{d-1}{2}n-d,
\end{eqnarray*}
where the second inequality holds since $4|S^{+}||S^{-}|\leq (|S^{+}|+|S^{-}|)^2$
and the last inequality holds since $|S^{+}|+|S^{-}|\leq n$.

If $\lambda_{2}(D)=\frac{d-1}{2}n-d$, then $|S^{+}|+|S^{-}|=n$,
$|S^{+}|=|S^{-}|$ and the equalities in Eqs. (\ref{eq1}) and (\ref{eq2})
hold. This implies that for any vertices $v,w\in S^-$,
$d(w,v)=d\geq 2$, and hence $v,w$ are nonadjacent; for
any vertex $v\in S^-$ and $u\in S^+$, $d(u,v)=1$, which
implies that $u$ and $v$ are adjacent. Thus,
$G[S^{-}]$  and similarly, $G[S^{+}]$ are independent sets and each
vertex in $S^{+}$ is adjacent to each vertex in $S^{-}$,
which implies that $G\cong K_{\frac{n}{2},\frac{n}{2}}$.

Conversely, it is routine to check that $\lambda_2(D(K_{\frac{n}{2},\frac{n}{2}}))=\frac{n}{2}-2$,
completing the proof.  {\hfill$\Box$}

\vspace{3mm}

Let $G$ be a graph and $v\in V(G)$. We denote by
$t(G,u)$ the number of triangles in $G$ containing the vertex $u$.

\vspace{3mm}
\noindent
{\bf Proof of Theorem \ref{ThIndepTri}.}
(i) Since $\alpha(G)\leq s$, we have $diam(G)\leq 2s-1$. By Theorem \ref{thLambda2Upper},
\begin{align}\label{Al_Th7-1}
\lambda_2(D)\leq (s-1)n-(2s-1).
\end{align}
Recall that a corollary of Tur\'an's inequality says that for any graph of order $n$
and size $m$, $\alpha(G)\geq \frac{n}{1+\overline{d}}$,
where $\overline{d}=\frac{2m}{n}$ (see Alon and Spencer \cite[pp.95]{AS}). Thus, we obtain
$s\geq \alpha(G)\geq \frac{n^2}{n+2m}$,
that is,
\begin{align}\label{Al_Th7-2}
m\geq \frac{n^2-sn}{2s}.
\end{align}

For each vertex $u\in V(G)$, $t(G,u)=e(G[N(u)])$. Note that $\alpha(G[N(u)])\leq s$.
Then (\ref{Al_Th7-2}) becomes
\begin{align}\label{Al_Th7-3}
t(G;u)\geq\frac{d^2(u)-sd(u)}{2s}.
\end{align}
Summing over all vertices for (\ref{Al_Th7-3}), we have
\begin{align}\label{Al_1}
3t(G)=\sum_{v\in V(G)}t(G;u)\geq\sum_{u\in V(G)}\frac{d^2(u)-sd(u)}{2s}=\frac{\sum_{v\in V(G)}d^2(u)}{2s}-m.
\end{align}
By AG-mean inequality,
\begin{align}\label{Al_2}
\frac{\sum_{v\in V(G)}d^2(u)}{2s}-m\geq \frac{(\sum_{v\in V(G)}d(u))^2}{2sn}-m=m(\frac{2m}{sn}-1)\geq m(\frac{n-s}{s^2}-1).
\end{align}
By (\ref{Al_1}) and (\ref{Al_2}), we obtain
\begin{align}\label{Al_Th7-4}
3s^3\cdot \frac{t(G)}{m}\geq s(n-s)-s^3>(s-1)n-(2s-1).
\end{align}
By (\ref{Al_Th7-1}) and (\ref{Al_Th7-4}), the proof is completed.

\vspace{3mm}

(ii) Setting $s=2$ in Theorem \ref{ThIndepTri}(i),
we have $\lambda_2(D(G))< \frac{24}{m}\cdot t(G)$
for $n>9$. By Tur\'an's theorem, when $\alpha(G)\leq 2$ and $n\geq 11$, we get
$$m\geq \frac{n^2-2n}{4}=\frac{n}{2}(\frac{n}{2}-1)\geq\frac{11}{2}\cdot \frac{9}{2}=24.75>24.$$
The proof is completed. {\hfill$\Box$}

\vspace{3mm}
\noindent
{\bf Proof of Proposition \ref{ThSkDiameter}.}
By Theorem \ref{ThRP} and Lemma \ref{LeLambdak-2}, we have $S_k(P_n)\geq \frac{n^2}{3}-2(k-1)$.
From Theorem \ref{ThZI}, we have
$\lambda_1(D)< dn-\frac{d(d-1)}{2}-1.$
By Theorems \ref{thLambda2Upper} and \ref{ThZI}, we have
\begin{eqnarray*}
S_k(G)&\leq& dn-\frac{d(d-1)}{2}-1+k(\frac{n(d-1)}{2}-d)\\
&=&(k\frac{d-1}{2}+d)n-\frac{d(d-1)}{2}-1-dk\\
&<&\frac{n^2}{3}-2(k-1)~~\mbox{(since $d<\frac{2n}{3(k+2)}$)}\\
&\leq&S_k(P_n).
\end{eqnarray*}
This completes the proof. {\hfill$\Box$}

\section{Concluding remarks}
It is known that $S_k(K_{r,n-r})=2n-2k$ for $k\geq 2.$ Theorem \ref{ThGeneralSk(G)} $(ii)$ shows that
$S_k(D)\geq 2n-2k=S_k(K_{1,n-1})$ if $G$ is a tree, so we may have the following more general problem.

\begin{prob}\label{prob1}
Let $G$ be a connected bipartite graph of order $n$. For an integer $k\geq2$
and sufficiently large $n$ with respect to $k$, does there always hold $S_{k}(G)\geq 2n-2k$,
where the equality holds if and only if $G\cong K_{r,n-r}$ for $1\leq r\leq n-1$?
\end{prob}

By Theorem \ref{thLambda2Upper}, we have $\lambda_2(D)\leq \frac{n(d-1)}{2}-d$. If $d=2$, then $\lambda_2(D)\leq \frac{n}{2}-2$.
It seems that the upper bound holds for every connected graph of order $n$, so we have the following problem.

\begin{prob}
Let $G$ be a connected graph with second largest distance eigenvalue $\lambda_{2}(D)$.
Then $\lambda_{2}(D)\leq \frac{n}{2}-2$ and the equality holds if and only if $G\cong K_{\frac{n}{2},\frac{n}{2}}$.
\end{prob}

\noindent\textbf{Acknowledgement}

\vspace{3mm}
The author would like to thank Bo Ning for sharing the proof of theorem \ref{ThIndepTri}.


\begin{thebibliography}{111}
\bibitem{AS}
N. Alon, J. H. Spencer, The probabilistic method. the third edition. Wiley-Interscience, New York.

\bibitem{AH16} M. Aouchiche, P. Hansen, Proximity, remoteness
and distance eigenvalues of a graph, \emph{Discrete Appl. Math.},
{\bf 213} (2016), 17--25.

\bibitem{AH14}
M. Aouchiche, P. Hansen, Distance spectra of graphs:
A survey, \emph{Linear Algebra Appl.}, {\bf 458} (2014), 301--386.

\bibitem{Bai11}
H. Bai, The Grone-Merris conjecture,
\emph{Trans. Amer. Math. Soc.}, {\bf 363} (2011), no. 8, 4463--4474.


\bibitem{BH}
A. E. Brouwer, W. H. Haemers, Spectra of graphs,
Universitext. Springer, New York, 2012. xiv+250 pp.

\bibitem{F98}
S. Fajtlowicz, Written on the wall: conjectures derived on
the basis of the program Galatea Gabriella Graffiti,
Technical report, University of Houston, 1998.

\bibitem{GM1}
R. Grone, R. Merris, Coalescence, majorization, edge valuations and the Laplacian
spectra of graphs,
\emph{ Linear Multilinear Algebra}, {\bf 27} No.2 (1990), 139--146.

\bibitem{GM2}
R. Grone, R. Merris, The Laplacian spectrum of a graph II,
\emph{SIAM J. Discrete Math.}, {\bf 7} (1994), 221--229.

\bibitem{HHL}
X. Huang, Q. Huang, L. Lu, Graphs with at Most Three Distance Eigenvalues Different from -1 and -2, \emph{Graphs Combin.},  {\bf 34}  (2018), 395--414.


\bibitem{L15-L}
H. Lin, Proof of a conjecture involving
the second largest $D$-eigenvalue and the number of
triangles, \emph{Linear Algebra Appl.}, {\bf 472} (2015),
48--53.

\bibitem{L15-D}
H. Lin, On the least distance eigenvalue and its applications on the
distance spread, \emph{Discrete Math.}, {\bf 338} (2015),
868--874.


\bibitem{LDW16}
H. Lin, K. Ch. Das, B. Wu, Remoteness
and distance eigenvalues of a graph,
\emph{Discrete Appl. Math.}, {\bf 215} (2016), 218--224.


\bibitem{LHH}
L. Lu, Q. Huang, X. Huang, The graphs with exactly two distance eigenvalues different from -1  and -3, \emph{J. Algebraic Combin.},  {\bf 45}  (2017), 629--647.



\bibitem{M90}
R. Merris, The distance spectrum of a tree,
\emph{J. Graph Theory}, {\bf 14} (1990), 365--369.

\bibitem{M09}
B. Mohar, On the sum of $k$ largest eigenvalues of graphs and symmetric matrices,
\emph{J. Combin. Theory, Ser. B}, {\bf 99} (2009), 306--313.


\bibitem{PHSM}
M. Pokorn\'{y}, P. H\'{i}c, D. Stevanovi\'{c}, M. Milo\v{s}evi\'{c}, On distance integral graphs, \emph{Discrete Math.},  {\bf 338}  (2015), 1784--1792.


\bibitem{R30}
F. P. Ramsey, On a problem of formal logic,
\emph{Proc. Lond. Math. Soc.}, {\bf } (1930),
264--286.

\bibitem{RP90}
S. Ruzieh, D. Powers, The distance spectrum of the path $P_n$
and the first distance eigenvector of connected graphs,
\emph{Linear Multilinear Algebra}, {\bf 28} (1990), 75--81.

\bibitem{ZC14}
F. Zhang, Z. Chen, Ramsey numbers, graph eigenvalues,
and a conjecture of Cao and Yuan,
\emph{Linear Algebra Appl.}, {\bf 458}
(2014), 526--533.

\bibitem{ZI10}
B. Zhou, A. Ili\'c, On distance spectral radius and
distance energy of graphs,
\emph{MATCH Commun. Math. Comput. Chem.}, {\bf 64} (2010),
261--280.






\end{thebibliography}
\end{document}